\documentclass[11pt]{article}

\usepackage{amsmath,amssymb,amsthm,setspace,graphicx,mathtools}
\usepackage[text={6.5in,8.4in}, top=1.3in]{geometry}
\usepackage{listings}

\newtheorem{theorem}{Theorem}[section]

\newtheorem{lemma}[theorem]{Lemma}
\newtheorem{claim}[theorem]{Claim}
\newtheorem{corollary}[theorem]{Corollary}
\newtheorem{conjecture}[theorem]{Conjecture}

\newtheorem{definition}[theorem]{Definition}

\begin{document}
	
\title{
		On factors of independent transversals in $k$-partite graphs
		\thanks{This research was supported by the Israel Science Foundation (grant No. 1082/16).}}
	
\author{
		Raphael Yuster
		\thanks{Department of Mathematics, University of Haifa, Haifa 3498838, Israel. Email: raphael.yuster@gmail.com}}
	
\date{}
	
\maketitle
	
\setcounter{page}{1}
	
\begin{abstract}
	
	A $[k,n,1]$-graph is a $k$-partite graph with parts of order $n$ such that the bipartite graph induced by any pair of parts is a matching.
	An independent transversal in such a graph is an independent set that intersects each part in a single vertex.
	A factor of independent transversals is a set of $n$ pairwise-disjoint independent transversals.
	Let $f(k)$ be the smallest integer $n_0$ such that every $[k,n,1]$-graph has a factor of independent transversals assuming $n \ge n_0$. Several known conjectures imply that for $k \ge 2$,
	$f(k)=k$ if $k$ is even and $f(k)=k+1$ if $k$ is odd.
	While a simple greedy algorithm based on iterating Hall's Theorem shows that $f(k) \le 2k-2$, no better bound is known and in fact, there are instances showing that the bound $2k-2$ is tight for the greedy algorithm.
	Here we significantly improve upon the greedy algorithm bound and prove that
	$f(k) \le 1.78k$ for all $k$ sufficiently large, answering a question of MacKeigan.
	
\vspace*{3mm}
\noindent
{\bf AMS subject classifications:} 05C35; 05C69\\
{\bf Keywords:} sparse partite graph; independent transversals; factor

\end{abstract}

\section{Introduction}

Given a $k$-partite graph, a {\em transversal} is a set of vertices containing a single vertex from each part. An {\em independent transversal\,} is an independent set which is also a transversal. A {\em factor of transversals} is a set of pairwise-disjoint transversals covering all vertices.
The problem of finding sufficient conditions for the existence of independent transversals and factors of independent transversals in $k$-partite graphs was studied by several researchers
\cite{ABZ-2007,alon-1992,BES-1975,catlin-1980,EGL-1994,fischer-1999,GS-2020arxiv,haxell-2004,KM-2015,mackeigan-2021,MM-2002,MS-2008,ST-2006,yuster-1997,yuster-1997b}
not least because it is strongly related to the Hajnal-Szemer\'edi Theorem and to concepts such as the strong chromatic number and list coloring.

In this paper we consider sufficient conditions for a factor of independent transversals in very sparse $k$-partite graphs. To this end, define a {\em $[k,n,1]$-graph} to be a $k$-partite graph with parts of order $n$ such that the bipartite graph induced by any pair of parts is a matching (more generally, in a $[k,n,\Delta]$-graph
the bipartite graph induced by any pair of parts has maximum degree at most $\Delta$).
Erd\H{o}s, Gy\'arf\'as and {\L}uczak \cite{EGL-1994} first considered $[k,n,1]$-graphs where the matching between each pair is a single edge, and asked for the smallest $n=n(k)$ such that there is always an independent transversal. They proved that $n \ge \sqrt{k}(1-o(1))$,
proved an upper bound which is larger by a constant factor, and conjectured that the lower bound is
asymptotically tight. The author \cite{yuster-1997} improved the upper bound of \cite{EGL-1994} and generalized the problem to arbitrary $[k,n,1]$-graphs. Very recently, the conjecture of \cite{EGL-1994} was solved independently by Glock and Sudkaov \cite{GS-2020arxiv} and by Kang and Kelly \cite{KK-2021}.

As for a sufficient condition guaranteeing a factor of independent transversals in $[k,n,1]$-graphs, the requirement of $n$ is obviously more demanding. Formalizing it, let $f(k)$ be the smallest integer $n_0$ such that every $[k,n,1]$-graph has a factor of independent transversals assuming $n \ge n_0$.
To see that $f(k) \ge k$, let the parts of a $[k,n,1]$-graph be denoted by $V_1,\ldots,V_k$ with
$V_i = \{v_{i,1},\ldots,v_{i,n}\}$. Let $G$ be the $[k,k-1,1]$-graph consisting of the edges
$v_{i,1}v_{j,1}$ for all $1 \le i < j \le k$. Then every independent transversal has a single
vertex of the form $v_{\star,1}$, so $G$ does not have a factor of independent transversals, proving that $f(k) \ge k$.
When $k \ge 3$ is odd, the following construction of Catlin \cite{catlin-1980} shows, in fact, that
$f(k) \ge k+1$. Let $G$ be the $[k,k,1]$-graph consisting of the edges
$v_{i,t}v_{j,t}$ for all $1 \le i < j \le k$ and for all $1 \le t \le k-2$
and of the edges $v_{i,k-1}v_{j,k}$ and $v_{i,k}v_{j,k-1}$ for all $1 \le i < j \le k$.
Then each independent transversal contains exactly one vertex of the form 
$v_{\star,t}$ for all $1 \le t \le k-2$ and exactly two vertices of the form $v_{\star,k-1}$ or exactly
two vertices of the form $v_{\star,k}$. But since $k$ is odd, this means that we can have at most $k-1$ disjoint
independent transversals, showing that $f(k) \ge k+1$.
As we shall see below, the following conjecture, which is implied by special cases of several known conjectures,
asserts that the above constructions are the worst cases.
\begin{conjecture}\label{conj:1}
	Let $k \ge 2$. Then $f(k)=k$ if $k$ is even and $f(k)=k+1$ if $k$ is odd.
\end{conjecture}
It is trivial that $f(2)=2$ and (see below) easy to show that $f(3)=4$. A computer-assisted proof (see Section 4)
shows that $f(4)=4$ but in general, Conjecture \ref{conj:1} is wide-open.
To motivate this conjecture, let us consider the more general case of $[k,n,\Delta]$-graphs.
Recall that the Hajnal-Szemer\'edi Theorem states that if a graph with $nk$ vertices has maximum degree
less than $n$, then it has $n$ pairwise-disjoint independent sets of size $k$ each.
Fischer \cite{fischer-1999} considered the analogue of the Hajnal-Szemer\'edi Theorem in
$[k,n,\Delta]$-graphs and raised the following intriguing conjecture: Every $[k,n,\Delta]$-graph
has a factor of independent transversals as long as $\Delta \le n/k$. Notice that if $k=n$ then
Fischer's Conjecture becomes equivalent to the conjecture $f(k)=k$. Fischer's conjecture is false when $k$ is odd due to Catlin's construction but, as we shall see evident, it stands a wide chance of being true in the sense that Catlin's construction (and its generalization for larger $\Delta$, see \cite{KO-2009}) is the only counter-example. In fact, this version (call it the ``modified Fischer's Conjecture'') is explicitly conjectured by K\"uhn and Osthus in \cite{KO-2009}. In particular, the modified Fischer's Conjecture implies Conjecture \ref{conj:1}.
The modified Fischer's Conjecture has been solved for $k=3$ and large $n$ by
Magyar and Martin \cite{MM-2002}, for $k=4$ and large $n$ by Martin and Szemer\'edi \cite{MS-2008}
and, finally, for every fixed $k$ and sufficiently large $n=n(k)$ by Keevash and Mycroft \cite{KM-2015}.
However, as all of these proofs require $n$ to be very large compared to $k$, they do not imply Conjecture
\ref{conj:1} nor an upper bound close to the conjecture.
Lastly, we mention that the author in \cite{yuster-1997} considered the refinement of $f(k)$ to $[k,n,1]$-graphs in which
every matching consists of at most $s$ independent edges (so the case $s=n$ is the general case
of $[k,n,1]$-graphs). Denoting the refined parameter by $f(k,s)$ we observe by the construction above
that already $f(k,1) \ge k$ and it is proved in \cite{yuster-1997} that $f(k,1)=f(k,2)=k$.
Furthermore, it is conjectured there that $f(k,s)=k$ for all $1 \le s \le k$, which again states that
$f(k)=k$, ignoring Catlin's obstacle (this was also observed in \cite{mackeigan-2021}), and so if the modified Fischer Conjecture holds (namely: Catlin's construction is the only obstacle), so does the modified conjecture on $f(k,s)=k$ unless $k \ge 3$ is odd and $s=k$ in which case $f(k,s)=k+1$.

There is, however, a simple upper bound of $f(k) \le 2k-2$ observed by MacKeigan \cite{mackeigan-2021}
which follows by greedily applying Hall's algorithm. Assume that $G$ is a $[k,n,1]$-graph with $n \ge 2k-2$.
Suppose we have already constructed a factor $F$ of independent transversals on the induced subgraph $G'$ of $G$ consisting of all but the last set $V_k$. Then $G'$ is a $[k-1,n,1]$-graph. Now consider the bipartite graph
$B$ with one part being $F$ and the other part being $V_k$, and there is an edge between $I \in F$ and $v \in V_k$ if $v$ is not a neighbor in $G$ of any vertex of $I$. Then $B$ has $n$ vertices in each part and
minimum degree at least $n-(k-1) \ge n/2$, so $B$ has a perfect matching by Hall's Theorem, implying that $F$ can be extended to a factor of independent transversals of $G$.
Notice that this already shows that $f(3) \le 4$ so together with Catlin's example, $f(3)=4$.
Unfortunately, there are examples where this greedy algorithm of repeatedly applying Hall's theorem fails to
produce a factor of independent transversals if $n=2k-3$.
Indeed we can construct examples of $F$ as above such that it has a subset of $k-1$ elements $F^*=\{I_1,\ldots,I_{k-1}\}$, and such that each of the vertices in $V^*=\{v_{k,1},\ldots,v_{k,k-1}\} \subset V_k$ has exactly one neighbor in $G$ belonging to $I_j$ for each $1 \le j \le k-1$ (for example, any Latin square of order $k-1$ corresponds to such a construction). Then in the bipartite graph $B$, the neighborhood of $V^*$ is only of size $n-(k-1)=k-2$, violating Hall's condition, so $B$ has no perfect matching.
A problem therefore raised by MacKeigan asks whether it is possible to improve upon the greedy algorithm.
Our main result shows that indeed this is the case.
\begin{theorem}\label{t:main}
	$f(k) \le 1.78k$ for all $k$ sufficiently large.
\end{theorem}
Our main idea in the proof of Theorem \ref{t:main} is to apply semi-randomness to the perfect matchings
at each stage of the greedy algorithm (hence it is no longer greedy) so as to guarantee that as we come close to the end of the algorithm (close to Stage $k$) we can still guarantee Hall's condition with high probability.

The rest of this paper continues as follows. In Section 2 we state and prove some preliminary lemmas that are
useful in the proof of the main result. Theorem \ref{t:main} is then proved in Section 3. In Section 4 we prove that
$f(4)=4$ thereby verifying Conjecture \ref{conj:1} in the first non-elementary case.
Throughout the paper, we assume, whenever necessary, that $k$ and $n$  are sufficiently large.

\section{Preliminary setup}

Let us first observe that if we can show that every $[k,n,1]$-graph has a factor of independent transversals, then $f(k) \le n$. Indeed, if $G$ is a $[k,n',1]$-graph with $n' > n$ we can just take an induced subgraph consisting of $n$ vertices of each part observing that this subgraph is an $[n,k,1]$-graph. Being such, it has a (factor of) independent transversal(s), so we can just apply this repeatedly to find $n'-n$ pairwise disjoint independent transversals in $G$. Removing them from $G$ we remain with an induced subgraph $G'$ of $G$ which is
an $[n,k,1]$-graph and find a factor of independent transversals in $G'$, hence, altogether, a factor of
independent transversals in $G$.

In our proof of Theorem \ref{t:main} it will be more convenient to fix the number of vertices in each part
and to parameterize on the number of parts. Let, therefore $g(n)$ denote the largest integer $k$ such that
every $[k,n,1]$-graph has a factor of independent transversals. We therefore have:
\begin{corollary}
	If $g(n) \ge 0.562 n$ for all $n$ sufficiently large, then $f(k) \le 1.78k$ for all $k$ sufficiently large,
	so Theorem \ref{t:main} holds.
\end{corollary}
\noindent As we prove in Lemma \ref{l:main} that $g(n) \ge 0.562 n$ for all $n$ sufficiently large,
we have by the last corollary that Theorem \ref{t:main} holds.

For a graph $G$ and a set of its vertices $X$, let $N_G(X)$ denote the set of vertices of $G$ that have a neighbor in $X$.
The following is an immediate corollary of Hall's Theorem.
\begin{lemma}\label{l:hall}
	Let $B$ be a bipartite graph with vertex parts $X,Y$ of order $m$ each. Assume that the degree of every vertex
	is at least $m-t$ and, furthermore, every $W \subseteq Y$ with $m-t \le |W| \le t$  has
	$|N_{B}(W)| \ge |W|$. Then $B$ has a perfect matching. \qed
\end{lemma}

\begin{lemma}\label{l:rand-match}
	Let $\delta > 0$. Then there exists a constant $C=C(\delta) > 0$ such that the following holds.
	Let $B$ be a bipartite graph with vertex parts $X=\{x_1,\ldots,x_n\},Y=\{y_1,\ldots,y_n\}$ and suppose the degree of each vertex of $X$ is precisely $n-t$.
	Furthermore, suppose that $X^* \subseteq X$ and $Y^* \subseteq Y$ are given.
	Let $\pi$ be a random permutation of the vertices of $Y$ and let $M=\{(x_i,y_{\pi(i)})\}_{i=1}^n$.
	Then with probability at least $1-C/n$ the following hold:\\
	(i) The number of pairs in $M$ that are edges of $B$ is at least $n-t-\delta n$
	and at most $n-t+\delta n$.\\
	(ii) There is a subset $M^* \subseteq M$ with $|M^*| \ge |X^*|(|Y^*|-t)/n-\delta n$
	such that every pair $(x_i,y_{\pi(i)}) \in M^*$ is an edge of $B$ and furthermore, $x_i \in X^*$ and $y_{\pi(i)} \in Y^*$.
\end{lemma}
\begin{proof}
	We start with (ii). Suppose without loss of generality, that the vertices of $X^*$ are $\{x_1,\ldots,x_q\}$ where $q=|X^*|$. Let $d_i$ be the number of neighbors of $x_i$ in $Y^*$ and observe that
	$d_i \ge (n-t)-(n-|Y^*|)=|Y^*|-t$.
	Let $R_i$ be the indicator random variable for the event $A_i$ = ``$y_{\pi(i)} \in Y^*$ and
	$(x_i,y_{\pi(i)})$ is an edge of $B$''. Let $R= \sum_{i=1}^{q} R_i$ count the number of
	pairs that satisfy the condition of being in $M^*$.
	Now, $Pr[A_i]=d_i/n \ge (|Y^*|-t)/n$.
	It follows that ${\mathbb E}[R] \ge q(|Y^*|-t)/n$.
	It remains to upper bound the probability that $R$ deviates from its expectation by more than
	$\delta n$. Now, if $q \le \delta n$ then (ii) is trivial since $R \le q$. Otherwise, $R$ is the sum of $\Theta(n)$ indicator random variables. We claim that $Var[R]=O(n)$.
	Indeed, $Var[R_i] \le 1$ as it is an indicator variable.
	We estimate $Cov(R_i,R_j)=\Pr[A_i \cap A_j]-d_id_j/n^2$.
	Now, given that $A_i$ occurred, there are two possibilities.
	If $y_{\pi(i)}$ is a neighbor of $x_j$, then $A_j$ occurred with probability $(d_j-1)/(n-1)$.
	If $y_{\pi(i)}$ is not a neighbor of $x_j$, then $A_j$ occurred with probability $d_j/(n-1)$.
	In any case,
	$$
	\frac{d_i(d_j-1)}{n(n-1)} \le \Pr[A_i \cap A_j] \le \frac{d_i d_j}{n(n-1)}
	$$
	So $Cov(R_i,R_j) \le \frac{2}{n-1}$ implying that $Var[R] = O(n)$.
	It now follows from Chebyshev's inequality that the probability that $R$ deviates from its expected value by more than $\delta n$
	is $O(1/n)$.
	
	For (i), apply the proof above for the special case that $X^*=X$ and $Y^*=Y$ (so now $q=n$
	and $|Y^*|=n$). Then the corresponding indicator random variable $R_i$ has $\Pr[R_i]=(n-t)/n$ since the degree of
	$x_i$ is precisely $n-t$. Hence, ${\mathbb E}[R] = n-t$ and, as we have shown, the probability that $R$ deviates from its expectation by more than $\delta n$ is $O(1/n)$.
\end{proof}

We need the next definition that formalizes the following procedure. Suppose we are given a bipartite graph $B$ with $n$ vertices in each part and a (not necessarily perfect) matching $M$ of $B$. If we take an induced subgraph $B^*$ of $B$ consisting of all the vertices not in $M$ and some pairs of matched vertices of $M$, we can consider a maximum matching in $B^*$. If this maximum matching of $B^*$ is a perfect matching of $B^*$, then (obviously) we obtain a perfect matching of $B$ as well.
\begin{definition}[$(B,M,s)$-reshuffling; leftover graph]\label{def:reshuffling}
	Let $n \ge s \ge 0$. Let $B$ be a bipartite graph with $n$ vertices in each part, and suppose that $M$ is a matching of $B$ consisting of $m \ge s$ edges. Consider an induced bipartite subgraph $B^*$ of $B$ obtained by taking all $2n-2m$ vertices not in $M$ and taking $2s$ vertices of $M$ obtained by randomly selecting $s$ edges of $M$ and taking their endpoints (observe that $B^*$ has $n-m+s$ vertices in each part). We call $B^*$ the {\em leftover graph}. Now take a maximum matching $M^*$ of $B^*$. Then $M^*$ together with the $m-s$ edges of $M$ not belonging to $B^*$ is a matching of $B$ of size $|M^*|+m-s$. In particular, if $M^*$ is a perfect matching of $B^*$ then we obtain a perfect matching of $B$ as well. We call this procedure a {\em $(B,M,s)$-reshuffling}. We say that the reshuffling is {\em successful} if $M^*$ is a perfect matching of $B^*$.
\end{definition}
\noindent We emphasize that when constructing $B^*$ in Definition \ref{def:reshuffling}, it will be very important to select the $s$ edges of $M$ at random, as stated, and not just arbitrarily.

In our proof we will need to ascertain that the constant we choose satisfies the following constraint.
\begin{lemma}\label{l:integeral}
	Let
	 $0 < c \le 1$ satisfy $2c^2\ln (\frac{1+c}{c}) \ge 1$ (note: any $0.778 \le c \le 1$ satisfies this).
	Then the following holds for all $0 \le \mu \le \frac{1}{1+c}$.
	$$
	\int_{0}^\mu \left[ 1-(1-c\mu-x)\left(1-\frac{cx}{1-x}\right)\right]dx \le c\mu\;.
	$$
	Furthermore, for every $\epsilon > 0$ there exists $\gamma > 0$ such that if $0 \le \mu \le \frac{1-\epsilon}{1+c}$, the definite integral is at
	most $(c-\gamma)\mu$.
\end{lemma}
\begin{proof}
	The definite integral equals $\mu^2(c^2+\frac{3c}{2}+\frac{1}{2})+\mu c^2\ln(1-\mu)$.
	Hence, dividing by $\mu$, it suffices to prove that the function
	$$
	f(\mu) = \mu(c^2+\frac{3c}{2}+\frac{1}{2})+c^2\ln(1-\mu)-c
	$$
	is non-positive in the range $0 \le \mu \le \frac{1}{1+c}$.
	The derivative of $f$ is
	$$
	f^{'}(\mu) = c^2+\frac{3c}{2}+\frac{1}{2} - \frac{c^2}{1-\mu}\;.
	$$
	So, $f$ is strictly monotone increasing if $\mu < 1-c^2/(c^2+3c/2+1/2)$ and indeed this is easily verified to hold for all $0 \le \mu \le \frac{1}{1+c}$. Since $f(0)=-c$, to prove the non-negativity of $f$ in the specified range it suffices to prove that $f(\frac{1}{1+c}) \le 0$.
	Now, rearranging the terms it is easily verified that
	$$
	f\left(\frac{1}{1+c}\right) \le 0 \quad \Leftrightarrow \quad 2c^2\ln \left(\frac{1+c}{c}\right) \ge 1\;.
	$$
	For the second part of the lemma, define $\gamma = h(\epsilon)=-f\left(\frac{1-\epsilon}{1+c}\right)$
	and observe that $h(0) \ge 0$ and that $h$ is a continuous strictly monotone increasing function, since $f$ is.
\end{proof}

\section{Proof of the main result}

In this section we assume that $G$ is a $[k,n,1]$-graph with vertex parts $V_1,\ldots,V_k$ of order $n$ each.
We let $G_t$ denote the subgraph of $G$ induced by $\cup_{i=1}^t V_i$, so $G_t$ is a $[t,n,1]$-graph.
\begin{definition}[$t$-partial factor of (independent) transversals]
A {\em $t$-partial factor of transversals of $G$} is a factor of transversals of $G_t$.
If each element in the factor is an independent set, it is called a
{\em $t$-partial factor of independent transversals of $G$}.
\end{definition}

In our proof we construct a sequence $F_1,\ldots,F_k$ such that $F_t$ is a $t$-partial factor of (hopefully independent) transversals of $G$. Let $F_t=\{I(t,1),\ldots,I(t,n)\}$ where $I(t,j)$ is a (not necessarily independent) transversal of $G_t$.
Our sequence will have the property that for every $1 \le t \le k-1$, $F_{t+1}$ {\em extends} $F_t$, meaning that $I(t+1,j) \setminus V_{t+1} = I(t,j)$.
To facilitate the construction of $F_{t+1}$ as an extension of $F_t$ we need the following definition.
\begin{definition}[The auxiliary bipartite graph $B_t$]\label{def:bt}
	Suppose that $F_t$ is $t$-partial factor of transversals.
	The {\em auxiliary bipartite graph} $B_t$ is defined by one part being $F_t$ (so the vertices of this part are $I(t,j)$ for $1 \le j \le n$) and the other part being $V_{t+1}$.
	There is an edge of $B_t$ connecting $I(t,j) \in F_t$ and $v \in V_{t+1}$ if all $t$ vertices of $I(t,j)$ are not neighbors of $v$ in $G$.
\end{definition}
Observe that a perfect matching in $B_t$ corresponds to $F_{t+1}$ by defining $I(t+1,j)$ to be the union of $I(t,j)$ and its match. Notice that indeed $F_{t+1}$ extends $F_t$ and notice that if $F_t$ is a $t$-partial factor of {\em independent} transversals, then $F_{t+1}$ is also a $(t+1)$-partial factor of independent transversals.
\begin{lemma}\label{l:deg-bt}
	The minimum degree of $B_t$ is at least $n-t$.
\end{lemma}
\begin{proof}
	Observe first that $I(t,i) \in F_t$ has $t$ vertices, and each such vertex has at most one neighbor (in $G$) belonging to $V_{t+1}$. Hence, the number of non-neighbors of $I(t,i)$ in
	$B_t$ is at most $t$, so the degree of $I(t,i)$ in $B_t$ is at least $n-t$.
	Observe next that each $v \in V_{t+1}$ has
	at most one neighbor (in $G$) in $V_j$ for $j=1,\ldots,t$. So, for each $1 \le j \le t$, there is at most one element of $F_t$ whose unique vertex belonging to $V_j$ is a neighbor (in $G$) of $v$. So, each $1 \le j \le t$ 
	contributes at most one non-neighbor of $v$ in $B_t$. Hence, the degree of $v$ in $B_t$ is at least $n-t$.
\end{proof}

\noindent The following is our main lemma.
\begin{lemma}\label{l:main}
	For all sufficiently large $n$ it holds that $g(n) \ge 0.562 n$.
\end{lemma}
\begin{proof}
	Let $0.778 \le c \le 1$ and recall that the conditions of Lemma \ref{l:integeral} hold.
	(Note: our proof works for every $c$ in this range, but we will eventually optimize by
	using $c=0.778$.)
	Suppose $\epsilon > 0$ is a given small constant.
	Let $\gamma=\gamma(\epsilon)$ be the constant from Lemma \ref{l:integeral} 
	and choose constants $0 < \delta \ll \eta \ll \gamma$ (by $x \ll y$ we mean that $x$ is a small function of $y$, small enough to satisfy the claimed inequalities that will follow). Let $C=C(\delta)$ be the constant from Lemma \ref{l:rand-match}.
	Throughout the proof we assume that $n$ is sufficiently large as a function of $\epsilon$.
	Let $k=\lfloor n(1-\epsilon)/(c+1) \rfloor$ and let $G$ be a $[k,n,1]$-graph with vertex parts $V_1,\ldots,V_k$ of order $n$ each. We will prove that $G$ has a factor of independent transversals.
	
	We construct partial factors of transversals $F_1,\ldots,F_k$ starting with $F_1=\{I(1,1),\ldots,I(1,n)\}$ where $I(1,j)$ is just a single vertex of $V_1$. Trivially, $F_1$ is a
	$1$-partial factor of independent transversals of $G$.
	We next define stages $t=1,\ldots,k-1$ where at Stage $t$ we construct (partly using a probabilistic argument) $F_{t+1}$ as an extension of $F_t$.
	Note that at this point we only assume that $F_t$ is a $t$-partial factor of transversals. However, we will prove that with positive probability, all elements of $F_t$, for all $t=1,\ldots,k$ are in fact independent sets. Hence we will
	obtain with positive probability that the final $F_k$ is a factor of independent transversals of $G$, as required.
	We now describe Stage $t$ in detail, assuming all previous stages have been completed (for completeness,
	Stage $0$ constructs the trivial $F_1$ as above).
	
	Consider $B_t$ in which one part is the already constructed $F_t=\{I(t,1),\ldots,I(t,n)\}$ and the other part is $V_{t+1}=\{y_1,\ldots,y_n\}$. Recall by Lemma \ref{l:deg-bt}  that each vertex of $B_t$ has degree at least $n-t$.
	So, remove edges from $B_t$ until each vertex of the part $F_t$ of $B_t$ has degree exactly $n-t$ and denote the resulting spanning subgraph by $B'_t$ (observe that in $B'_t$ some vertices in the part $V_{t+1}$ may now have degree smaller than $n-t$).
	Randomly select a permutation of $V_{t+1}$, denoting it by $\pi_{t}$.
	Now consider the set of all pairs $M(\pi_{t})=\{(I(t,i),y_{\pi_t(i)})\}_{i=1}^n$.
	Unfortunately, not each pair $(I(t,i),y_{\pi_t(i)})$ is necessarily an edge of $B_t$ and thus, moreover, not necessarily an edge of $B'_t$ (recall, to be an edge of $B_t$ we should have that all $t$ vertices of $I(t,i)$ are not neighbors in $G$ of $y_{\pi_t(i)}$).
	Let $M'(\pi_{t})$ denote the set of elements of $M(\pi_{t})$ that are edges of $B'_t$ and set $m_t=|M'(\pi_{t})|$.
	We say that $M(\pi_{t})$ is {\em good} if $|m_t-(n-t)| \le \delta n$.
	\begin{claim}\label{claim:1}
		$M(\pi_{t})$ is good with probability at least $1-C/n$.
	\end{claim}
	\begin{proof}
		We apply Lemma \ref{l:rand-match}
		with $B=B'_t$, $X=F_t$, $Y=V_{t+1}$ and $\pi=\pi_t$.
		According to case (i) of that lemma, with probability at least $1-C/n$,  $M(\pi_{t})$ contains
		at least $(n-t)-\delta n$ edges of $B'_t$ and at most $(n-t)+\delta n$ edges of $B'_t$.
	\end{proof}
	It is important to note that the stated probability $1-C/n$ of $M(\pi_t)$ being good does not depend on the goodness of $M(\pi_{\ell})$ for any $\ell < t$, as the proof of Lemma \ref{l:rand-match} does not assume anything except that the degree of each vertex of $B'_t$ is $n-t$ and that the permutation $\pi_t$ is chosen at random.

	Now, if $M(\pi_{t})$ is good (which happens with very high probability by Claim \ref{claim:1})
	we proceed by
	performing a $(B_t,M'(\pi_t),s_t)$-reshuffling
	where $s_t = \lfloor ct+\eta n \rfloor$ (recall Definition \ref{def:reshuffling}).
	To see that the parameters in Definition \ref{def:reshuffling} fit, we must show that
	$s_t \le m_t = |M'(\pi_t)|$. Indeed this holds since
	$s_t \le ct+\eta n \le n-t-\delta n \le m_t$ where we have used that $t(c+1) \le k(c+1) \le  n(1-\epsilon) \le n(1-\delta-\eta)$.
	Now, if this $(B_t,M'(\pi_t),s_t)$-reshuffling is successful (recall again Definition \ref{def:reshuffling}),
	then $B_t$ has a perfect matching, and we define $F_{t+1}$ using this perfect matching as shown in the paragraph after Definition \ref{def:bt}. Recall that $F_{t+1}$ extends $F_t$ and that if $F_t$ is a $t$-partial factor of {\em independent} transversals, then $F_{t+1}$ is also a $(t+1)$-partial factor of independent transversals.
	We can now define the notion of a successful stage.
	\begin{definition}[Successful stage]\label{def:successful}
		We say that Stage $t$ is {\em successful} if $M(\pi_{t})$ is good and furthermore, the corresponding $(B_t,M'(\pi_t),s_t)$-reshuffling is successful.
	\end{definition}
	\noindent 	As an immediate corollary we obtain:
	\begin{corollary}\label{coro:1}
		If all stages are successful then $F_k$ is a factor of independent transversals of $G$. \qed
	\end{corollary}
	Now, Corollary \ref{coro:success} below asserts that indeed, the probability that all stages are successful is positive, so given Corollary \ref{coro:success}, we have that $F_k$ is a factor of independent transversals of $G$. Now, if we choose $c=0.778$ then $k \ge 0.5624 n(1-\epsilon)$, so Lemma \ref{l:main} holds.
\end{proof}

It remains to prove Corollary \ref{coro:success} below, namely that with positive probability, all stages are successful. To help with the analysis, it would be
convenient (i.e. more uniform) to also say how we construct $F_{t+1}$ as an extension of $F_t$ in case that Stage $t$ is
not successful, either because $M(\pi_{t})$ is not good or because the $(B_t,M'(\pi_t),s_t)$-reshuffling failed.
While this may seem artificial (since if some stage failed, why proceed to the next stage?) the analysis becomes more uniform as it will help us lower-bound success probability of stages regardless of successes or failures of previous stages.
So, if Stage $t$ is unsuccessful, we simply use $M(\pi_{t})$ (which might certainly not be a perfect matching of $B_t$) to define $F_{t+1}$.
In other words, we define $I(t+1,i) = I(t,i) \cup \{y_{\pi_t(i)}\}$ for $i=1,\ldots,n$ and $F_{t+1}=\{I(t+1,i)\}_{i=1}^n$. Observe that even if $F_t$ is a factor of independent transversals, $F_{t+1}$ might not be, since if $(I(t+1,i),y_{\pi_t(i)})$ is not an edge of $B_t$, then $I(t+1,i)$ is not an independent set of $G_{t+1}$.
Nevertheless, $F_{t+1}$ is still a factor of transversals that extends $F_t$.
We have now defined how to perform all stages $1 \le t \le k-1$ (some of which may be successful, while others might not be).
\begin{corollary}\label{cor:all-good}
	With probability at least $\frac{1}{2}e^{-C}$ we have that $M(\pi_{t})$ is good for all stages $1 \le t \le k-1$.
\end{corollary}
\begin{proof}
	By Claim \ref{claim:1}, for each $t$ it holds that $M(\pi_t)$ is good with probability $1-C/n$
	and recall that this probability bound holds for stage $t$ regardless of the goodness of stages other than $t$. Hence, the probability that  $M(\pi_{t})$ is good for all $1 \le t \le k-1$ is at least
	$(1-C/n)^k \ge (1-C/n)^n \ge \frac{1}{2}e^{-C}$.
\end{proof}

Given that $M(\pi_{t})$ is good, we would like to ascertain some property whose existence guarantees that the corresponding $(B_t,M'(\pi_t),s_t)$-reshuffling is successful. One way to do that is to prove that the leftover graph ${B_t}^*$ of the $(B_t,M'(\pi_t),s_t)$-reshuffling
satisfies the conditions of Lemma \ref{l:hall} (so $B_t^*$ has a perfect matching implying that the $(B_t,M'(\pi_t),s_t)$-reshuffling is successful, hence the entire Stage $t$ is successful). We start by showing that if $t$ is very small (namely, in the first few stages) then, given that $M(\pi_t)$ is good, surely (i.e. with probability $1$)
the corresponding $(B_t,M'(\pi_t),s_t)$-reshuffling is successful.
\begin{lemma}\label{l:small}
	For all $1 \le t \le \eta n/3$, if $M(\pi_{t})$ is good, then the corresponding $(B_t,M'(\pi_t),s_t)$-reshuffling is successful, namely Stage $t$ is successful.
\end{lemma}
\begin{proof}
	We need to prove that the leftover graph $B_t^*$ has a perfect matching.
	Recall (Definition \ref{def:reshuffling}) that $B_t^*$ has $n-m_t+s_t$ vertices in each part.
	On the other hand, $B_t^*$ is an induced subgraph of $B_t$ and by Lemma \ref{l:deg-bt} the latter has minimum degree at least $n-t$. Hence the minimum degree of $B_t^*$ is at least $n-m_t+s_t-t$. Since by Hall's Theorem, every bipartite graph with $x$ vertices in each side and minimum degree at least $x/2$ has a perfect matching, it suffices to prove that $n-m_t+s_t-t \ge (n-m_t+s_t)/2$. Equivalently, it suffices to prove that
	$t \le (n-m_t+s_t)/2$. Indeed, this holds since
	$$
	\frac{n-m_t+s_t}{2} \ge \frac{s_t}{2} = \frac{\lfloor ct+\eta n \rfloor}{2} \ge \frac{\eta n}{3} \ge t\;.
	$$
\end{proof}
In contrast with Lemma \ref{l:small}, proving that (with high probability) the reshuffling corresponding to larger $t$ is successful is more involved.
We start with the following lemma that gives a sufficient condition for a reshuffling to be successful
when $t \ge \eta n/3$.
\begin{lemma}\label{l:success}
	Let $t \ge \eta n/3$ and suppose that $M(\pi_{t})$ is good.
	If each $W \subset V_{t+1}$ with $|W| = \lfloor ct \rfloor$ has $|N_{B_t}(W)| \ge n-\lfloor \eta n /7 \rfloor$
	then the $(B_t,M'(\pi_t),s_t)$-reshuffling is successful.
\end{lemma}
\begin{proof}
	As in the proof of Lemma \ref{l:small} we observe that $B_t^*$ has $n-m_t+s_t$ vertices in each part
	and its minimum degree is at least $n-m_t+s_t-t$. Let $W \subseteq V_{t+1}$ with $t \ge |W| \ge n-m_t+s_t-t$.
	If we can show that $|N_{B_t^*}(W)| \ge |W|$ then by Lemma \ref{l:hall}, $B_t^*$ has a perfect matching.
	It thus suffices to prove the stronger statement, that each $W \subseteq V_{t+1}$ with $|W|=n-m_t+s_t-t$ has
	$|N_{B_t^*}(W)| \ge t$. To prove the latter, it suffices to prove the even stronger statement that
	$|N_{B_t}(W)| \ge t+m_t-s_t$.
	Now, as $M(\pi_{t})$ is good we have that $m_t \le n-t+\delta n$ and recall that
	$s_t=\lfloor ct+\eta n \rfloor \ge ct$. Thus, we have:
	\begin{align*}
	t+m_t-s_t & \le t+(n-t)+\delta n - ct \\
	& = n+\delta n - ct \le n+\delta n - c\eta n/3 \\
	& \le n +\delta n -\eta n/6 \\
	& \le n-\lfloor \eta n /7 \rfloor
	\end{align*}
	where we have used that $c \ge \frac{1}{2}$ and $\delta \ll \eta$.
	We also have that
	\begin{align*}
	n-m_t+s_t-t & \ge n-(n-t+\delta n)+s_t-t \\
	& = s_t-\delta n \\
	& = \lfloor ct+\eta n \rfloor -\delta n \\
	& \ge \lfloor ct \rfloor\;.
	\end{align*}
	Hence, if $W \subset V_{t+1}$ with $|W| = \lfloor ct \rfloor$ has $|N_{B_t}(W)| \ge n-\lfloor \eta n /7 \rfloor$ then $B_t^*$  has a perfect matching, implying that the $(B_t,M'(\pi_t),s_t)$-reshuffling is successful.
\end{proof}

Proving that the conditions of Lemma \ref{l:success} hold with high probability, namely that with high probability each $W \subset V_{t+1}$ with $|W| = \lfloor ct \rfloor$ has $|N_{B_t}(W)| \ge n-\lfloor \eta n /7 \rfloor$, is rather technical and requires a few additional definitions, notations, and lemmas.

For $Q \subseteq [n]$ and for a $t$-partial factor of transversals $F_t=\{I(t,i)\}_{i=1}^n$, let
$F_t(Q) = \{I(t,i)\;|\; i \in Q\}$. For $W \subseteq V_{t+1}$ and $1 \le \ell \le t$, let
$Y_\ell(W)=V_\ell \setminus N_{G}(W)$ (i.e. the non-neighbors of $W$ in $G$ that belong to $V_\ell$).
Also define $Y_{t+1}(W)=V_{t+1} \setminus W$. Observe that $|Y_\ell(W)| \ge n-|W|$ holds for all $1 \le \ell \le t+1$ since each vertex of $W$ has at most one neighbor in $G$ that belongs to $V_{\ell}$.

We need the notions of $F_t(Q)$ and $Y_\ell(W)$ from the previous paragraph in the following definition, which 
is an important ingredient in the remainder of the proof.
\begin{definition}[$(Q,\ell,W,r)$-intersection]\label{def:intersection}
	Let $Q \subseteq [n]$, let $W \subseteq V_{t+1}$ and let $1 \le \ell \le t$.
	We say that a {\em $(Q,\ell,W,r)$-intersection} exists if\\
	(i) $M(\pi_\ell)$ is good (so we proceed to doing a  $(B_\ell,M'(\pi_\ell),s_\ell)$-reshuffling),\\
	(ii) there are at least $r$ edges of $M'(\pi_\ell)$ where in each of these $r$ edges,
	one endpoint is from $F_\ell(Q)$ and the other endpoint is from $Y_{\ell+1}(W)$,\\
	(iii) the $2r$ endpoints of these $r$ edges from (ii) are {\em not} selected to the leftover graph $B_\ell^*$.
\end{definition}

\noindent Until the end of this section, let
$$
q=\lceil \eta n/7 \rceil \quad \quad   r_{\ell,t} = \frac{q(n-ct-\ell)}{n}\left(1-\frac{c \ell}{n-\ell}\right)-5\eta q\;.
$$
\begin{lemma}\label{l:intersection}
	Let $1 \le t \le k-1$, let $T=\{1,\ldots,t\}$, let $W \subseteq V_{t+1}$ with $|W|=\lfloor ct \rfloor$, and let $Q \subset [n]$ with $|Q|=q$. Then with probability at least $1-5^{-n}$ the following holds.
	There exists $T^* \subseteq T$ with $|T^*| \ge t-\delta n$ such that for all $\ell \in T^*$,
	a $(Q,\ell,W,r_{\ell,t})$-intersection exists.
\end{lemma}
\begin{proof}
	Suppose that $t,W,Q$ are given as in the statement of the lemma. Fixing some $\ell \in T$ we estimate the probability that a $(Q,\ell,W,r_{\ell,t})$-intersection exists. We recap what we are doing in Stage $\ell$.
	At the beginning of that stage, we have $F_\ell$, the already constructed $\ell$-partial factor of transversals
	(note: we do not assume anything about success or failure of prior stages) as one part of $B_\ell$ and we 
	have $V_{\ell+1}$ as the other part. We remove some edges of $B_\ell$ to obtain $B'_\ell$. We then take a random permutation $\pi_\ell$ of $V_{\ell+1}$ yielding $M(\pi_{\ell})$ and its subset $M'(\pi_\ell)$ which is a matching of
	$B'_\ell$. Observe that $F_\ell(Q)$ is a subset of $F_\ell$ of order $q$, and that $Y_{\ell+1}(W)$ is a subset of
	$V_{\ell+1}$ of order at least $n-|W|$.
	
	Now, apply Lemma \ref{l:rand-match} with $B=B'_t$, $X=F_\ell$,
	$Y=V_{\ell+1}$, $X^*=F_\ell(Q)$, $Y^*=Y_{\ell+1}(W)$, $\pi=\pi_\ell$, $M=M(\pi_{\ell})$.
	We obtain from that lemma that with probability $1-C/n$, both items (i) and (ii) of Lemma \ref{l:rand-match}  hold.
	Recall that item (i) of Lemma \ref{l:rand-match} says that $M(\pi_\ell)$ is good, since $|m_\ell-(n-\ell)| \le \delta n$ where $m_\ell=|M'(\pi_\ell)|$, so already item (i) of Definition \ref{def:intersection} holds.
	Item (ii) of Lemma \ref{l:rand-match} says that there is a subset $M^* \subset M(\pi_\ell)$ of size
	\begin{equation}\label{e:m*}
	|M^*| \ge \frac{|F_\ell(Q)|(|Y_{\ell+1}(W)|-\ell)}{n}-\delta n  \ge \frac{q(n-|W|-\ell)}{n} -\delta n \ge
	\frac{q(n-ct-\ell)}{n} -\delta n
	\end{equation}
	such that every pair $(x_i,y_{\pi(i)}) \in M^*$ is and edge of $B'_t$ (hence an edge of $M'(\pi_\ell)$) and furthermore, $x_i \in F_\ell(Q)$ and $y_{\pi_\ell(i)} \in Y_{\ell+1}(W)$.
	
	Notice that if $r_{\ell,t} \ge |M^*|$ then Item (ii) of Definition \ref{def:intersection} holds, but this is not enough for us, as we also want Item (iii) to hold. In the reshuffling, we create $B^*_\ell$ by randomly selecting $s_\ell$ out of the $m_\ell$ edges of $M'(\pi_\ell)$ and add the $2s_\ell$ endpoints of the selected edges to the vertices not matched by $M'(\pi_\ell)$. So, the probability that a single edge (and its two endpoints) from $M^*$ is selected to $B^*_\ell$ is precisely $s_\ell/m_\ell$.
	Hence, if $Z \subseteq M^*$ denotes the set of edges of $M^*$ that are {\em not} selected to the leftover graph $B^*_\ell$ then ${\mathbb E}[|Z|]=|M^*|(1-\frac{s_\ell}{m_\ell})$. Observe that $|Z|$ has hypergeometric distribution
	(as its distribution is identical to having $|M^*|$ red balls, $m_t-|M^*|$ blue balls, and we select without replacement $m_\ell-s_\ell$ balls and ask for the number of red balls that are selected).
	As trivially, $|Z| \le n$, we have by Chebyshev's inequality,
	that the probability that $|Z|$ deviates from its expected value by more than $\delta n$ is $O(1/n)$ (since 
	the variance in our hypergeometric distribution is $O(n)$). So overall, we obtain that with probability $1-O(1/n)$, 
	$$
	|Z| \ge |M^*|\left(1-\frac{s_\ell}{m_\ell}\right) - \delta n\;.
	$$
	To prove that with probability $1-O(1/n)$, a $(Q,\ell,W,r_{\ell,t})$-intersection exists we just need to make sure that $r_{\ell,t}$ is not larger than the right-hand side of the last inequality.
	Indeed, using (\ref{e:m*}), using Lemma \ref{l:ratio} below which upper bounds $s_\ell/m_\ell$ and using $\delta \ll \eta$ we obtain
	\begin{align*}
		|M^*|\left(1-\frac{s_\ell}{m_\ell}\right) - \delta n & \ge \left(\frac{q(n-ct-\ell)}{n} -\delta n\right)\left(1-\frac{s_\ell}{m_\ell}\right) - \delta n \\
		& \ge \frac{q(n-ct-\ell)}{n}\left(1-\frac{s_\ell}{m_\ell}\right) - 2\delta n \\
		& \ge \frac{q(n-ct-\ell)}{n}\left(1-\frac{c \ell}{n-\ell}-4\eta\right) - 2\delta n \\
		& \ge \frac{q(n-ct-\ell)}{n}\left(1-\frac{c \ell}{n-\ell}\right)-4\eta q  - 2\delta n  \\
		& \ge \frac{q(n-ct-\ell)}{n}\left(1-\frac{c \ell}{n-\ell}\right)-5\eta q    \\
		& = r_{\ell,t}\;.
	\end{align*}
	
	We have proved that with probability $1-O(1/n)$, a $(Q,\ell,W,r_{\ell,t})$-intersection exists but we still need to prove that the claimed $T^*$ exists with high probability. First observe that if $t \le \delta n$ then
	$T^*$ exists vacuously so assume that $t \ge \delta n$.
	For $T^*$ not to exist, there should be a set of $\lceil \delta n \rceil$ indices $\ell$ with $1 \le \ell \le t$, such that a $(Q,\ell,W,r_{\ell,t})$-intersection does not exist.
	Fix a set $L \subseteq T$ of $\lceil \delta n \rceil$ indices. We have proved that the probability that for $\ell \in L$,  a $(Q,\ell,W,r_{\ell,t})$-intersection does not exist is $O(1/n)$ regardless of the outcome of any other $\ell' \in L$. Hence, the probability that all $\ell \in L$ are such that the corresponding $(Q,\ell,W,r_{\ell,t})$-intersection does not exist is $O(1/n)^{\delta n}$.
	As there are less than $2^n$ possible sets $L$ to consider, we have by the union bound that the probability that $T^*$ does not exist is at most
	\[
	2^n O(1/n)^{\delta n} \ll 5^{-n}\;. \qedhere
	\]
\end{proof}	
We next prove the upper bound for $\frac{s_\ell}{m_\ell}$ that we have used in Lemma \ref{l:intersection}.
\begin{lemma}\label{l:ratio}
	For all $1 \le \ell \le k-1$, if $M(\pi_\ell)$ is good then
	$$
	\frac{s_\ell}{m_\ell} \le \frac{c\ell}{n-\ell}+ 4\eta\;.
	$$
\end{lemma}
\begin{proof}
	Recall that $s_\ell = \lfloor c\ell + \eta n \rfloor$ and that $M(\pi_\ell)$ being good implies that
	$m_\ell \ge n-\ell -\delta n$. Also notice that $\ell < k < n/(1+c) \le 2n/3$ and that $\delta \ll \eta$. Hence, 
	\begin{align*}
		\frac{s_\ell}{m_\ell} &  \le \frac{c\ell + \eta n}{n-\ell-\delta n}\\
		& = \frac{c\ell-\delta n}{n-\ell-\delta n}+ \frac{(\delta+\eta) n}{n-\ell-\delta n}\\
		& \le \frac{c\ell}{n-\ell}+ \frac{(\delta+\eta) n}{n-\ell-\delta n}\\
		& \le \frac{c\ell}{n-\ell}+ \frac{(\delta+\eta) n}{n/3-\delta n}\\
		& \le \frac{c\ell}{n-\ell}+ 4\eta\;.
	\end{align*}
\end{proof}

Lemma \ref{l:sufficient} below proves that if all stages satisfy a certain condition, then all stages are successful. The next lemma describes this condition and proves that it holds with positive probability.
\begin{lemma}\label{l:highprob}
	Let ${\cal A}$ denote the intersection of the two following events:\\
	(i) $M(\pi_{t})$ is good for all stages $1 \le t \le k-1$.\\
	(ii) For all $1 \le t \le k-1$, for all $W \subseteq V_{t+1}$ with $|W|=\lfloor ct \rfloor$, and for all
	$Q \subset [n]$ with $|Q|=q$ there exists $T^* \subseteq \{1,\ldots,t\}$ with $|T^*| \ge t-\delta n$ such that for all $\ell \in T^*$, a $(Q,\ell,W,r_{\ell,t})$-intersection exists.\\
	Then ${\cal A}$ holds with positive probability.
\end{lemma}
\begin{proof}
	By Corollary \ref{cor:all-good}, (i) holds with probability at least $\frac{1}{2}e^{-C}$.
	As for (ii), observe that the number of choices for $Q$ is less than $2^n$ and that for a given $t$, the number of possible choices for  $W$ is less than $2^n$.
	Hence, the number of possible triples $(t,Q,W)$ is
	less than $n4^n$. For each such triple, the probability that there exists $T^*$ as specified is at least
	$1-5^{-n}$ by Lemma \ref{l:intersection}. So, the probability that (ii) is satisfied is at least
	$1-n4^n5^{-n} >1- 1/n$. Hence the event ${\cal A}$ holds with probability at least $\frac{1}{2}e^{-C}-1/n > 0$.
\end{proof}	

\begin{lemma}\label{l:sufficient}
	If ${\cal A}$ holds then all stages are successful.
\end{lemma}
\begin{proof}
	Assume by contradiction that ${\cal A}$ holds, yet there is an unsuccessful stage $t$.
	First observe that we must have $t \ge \eta n/3$, since by Lemma \ref{l:small}, for smaller $t$,
	Stage $t$ is successful as $M(\pi_{t})$ is good (since we assume that ${\cal A}$ holds).
	
	 As $M(\pi_{t})$ is good, the only reason that Stage $t$ is unsuccessful is because the 
	 $(B_t,M'(\pi_t),s_t)$-reshuffling is not successful. Hence, by Lemma \ref{l:success}, there exists
	 $W \subset V_{t+1}$ with $|W| = \lfloor ct \rfloor$ such that $|N_{B_t}(W)| < n-\lfloor \eta n /7 \rfloor$.
	 This means that there are at least $q=\lfloor \eta n /7 \rfloor$ elements of $F_t$ that are non-neighbors in $B_t$ of all elements of $W$.
	 Let, therefore, $Q \subset [n]$ with $|Q|=q$ be such that $F_t(Q)=\{I(t,i)\,|\, i \in Q\} \subset F_t$ and
	 for each pair $(I(t,i),y)$ where $I(t,i) \in F_t(Q)$ and $y \in W$, $I(t,i)$ and $y$ are non-adjacent in $B_t$. Being non-adjacent in $B_t$ means that for each $y \in W$ and for each
	 $I(t,i) \in F_t(Q)$, there is at least one vertex $x \in I(t,i)$ such that $xy$ is an edge of $G$.
	 Let $P = \cup_{i \in Q} I(t,i)$ and let $P^* \subseteq P$ be such that $x \in P^*$ if and only if
	 there exists $y \in W$ such that $xy$ is an edge of $G$. So we have that $|P|=qt$ and
	 \begin{equation}\label{e:v1}
	 |P^*| \ge q|W| = q \lfloor ct \rfloor \ge q(ct-1)
	 \end{equation}
	 since every $y \in W$ has at least $q$ neighbors in $G$ that belong to $P^*$, and these neighbors are unique for $y$ since $G$ is an $[n,k,1]$-graph (every vertex in $P^*$ has exactly one neighbor in $W$).
	 We will now count the number of elements of $P^*$ in another way, and obtain that it is less than
	 $q(ct-1)$, hence reaching a contradiction with (\ref{e:v1}).
	 
	 For every $1 \le \ell \le t$, let $P_\ell = P \cap V_\ell$ and $P^*_\ell = P^* \cap V_\ell$.
	 Clearly $|P_\ell| = q$ so $|P^*_\ell| \le q$, but we can obtain a better upper bound for $|P^*_\ell|$
	 for most values of $\ell$. In order to show this, recall that
	 as we assume that ${\cal A}$ holds, there exist  $T^* \subseteq \{1,\ldots,t\}$ with $|T^*| \ge t-\delta n$ such that for all $\ell \in T^*$, a $(Q,\ell,W,r_{\ell,t})$-intersection exists.
	 
	 Consider first $\ell \in T \setminus T^*$. For such $\ell$ we will use the trivial bound $|P^*_{\ell+1}| \le q$
	 but since there are at most $\delta n$ such $\ell$, we have that
	 $$
	 \sum_{\ell \in T \setminus T^*} |P^*_{\ell+1}| \le \delta n q\;.
	 $$
	 
	 Consider next $\ell \in T^*$ in which case a $(Q,\ell,W,r_{\ell,t})$-intersection exists.
	 First observe that by the definition of $P_{\ell+1}$, when we perform  Stage $\ell$ extending $F_\ell$
	 to $F_{\ell+1}$, the elements of $F_\ell(Q)$ are matched with the vertices of $P_{\ell+1}$ to obtain $F_{\ell+1}(Q)$.
	 Let us recall what it means for our $Q,W$ that a $(Q,\ell,W,r_{\ell,t})$-intersection exists. By Definition \ref{def:intersection} applied to Stage $\ell$, $Q$ and $W$, we see that Case (i) of that definition holds since ${\cal A}$ holds.
	 Now, Cases (ii) and (iii) of Definition \ref{def:intersection} say that when we extend $F_\ell(Q)$ to $F_{\ell+1}(Q)$, at least $r_{\ell,t}$ elements of $F_\ell(Q)$ are matched to vertices of $Y_{\ell+1}(W)$
	 and notice that $Y_{\ell+1}(W) \cap P^*_{\ell+1} = \emptyset$ since the elements in $Y_{\ell+1}(W)$ are non-neighbors in $G$ of the vertices of $W$ while the elements of $P^*_{\ell+1}$ are neighbors in $G$ of the vertices of $W$. Hence we have
	 $$
	 |P_{\ell+1} \setminus P^*_{\ell+1}| \ge r_{\ell,t}
	 $$
	 or equivalently,
	 $$
	 |P^*_{\ell+1}|  \le q- r_{\ell,t} \le q- \frac{q(n-ct-\ell)}{n}\left(1-\frac{c \ell}{n-\ell}\right)+5\eta q\;.
	 $$
	 Altogether, we have
	 $$
	 |P^*| = \sum_{\ell=1}^t |P^*_\ell| \le |P_1|+ \delta n q + \sum_{\ell=1}^{t-1} \left(q- \frac{q(n-ct-\ell)}{n}\left(1-\frac{c \ell}{n-\ell}\right)+5\eta q\right)\;.
	 $$
	 Observe that in the last inequality we are summing over all $1 \le \ell \le t-1$ although it would have been enough to sum over the elements of $T^*$. To obtain the desired contradiction to (\ref{e:v1}) it therefore suffices to prove that
	 $$
	 q+\delta n q + \sum_{\ell=1}^{t-1} \left(q- \frac{q(n-ct-\ell)}{n}\left(1-\frac{c \ell}{n-\ell}\right)+5\eta q \right) < q(ct-1)\;.
	 $$
	 Dividing both sides by $q$, our remaining task is to prove
	 $$
	 \delta n + \sum_{\ell=1}^{t-1} \left(1- \frac{(n-ct-\ell)}{n}\left(1-\frac{c \ell}{n-\ell}\right)+5\eta \right) < ct-2\;.
	 $$
	 Rearranging the terms, it suffices to show that
	 $$
	 \sum_{\ell=1}^{t-1} \left(1- \frac{(n-ct-\ell)}{n}\left(1-\frac{c \ell}{n-\ell}\right) \right) < ct-2-\delta n-5\eta t\;.
	 $$
	 Recall that $t \ge \eta n/3$ and that $\delta \ll \eta$ so $ct-2-\delta n-5\eta t > t(c-6\eta)$.
	 Now, define $\mu = t/n$ and define $x=\ell/n$. So it suffices to prove that
	 $$
	 \frac{1}{n} \left [ \sum_{x\in\{\frac{\ell}{n}\,|\, 1 \le \ell \le t-1\}} \left(1- (1-c\mu-x)\left(1-\frac{cx}{1-x}\right) \right) \right] \le \mu(c-6\eta).
	 $$
	 As the (real) function $f(x) = 1- (1-c\mu-x)\left(1-\frac{cx}{1-x}\right)$ is monotone increasing 
	 in $[0,\mu]$ and $f(x) \in [0,1]$ in that range, the sum in square brackets is at most $n$ times the corresponding integral of $f(x)$ in the range $[0,\mu]$, hence it suffices to prove that
	 $$
	 \int_{0}^\mu \left[ 1-(1-c\mu-x)\left(1-\frac{cx}{1-x}\right)\right]dx \le \mu(c-6\eta).
	 $$
	 But now $\mu =t/n \le k/n \le \frac{1-\epsilon}{1+c}$ and $\eta \le \gamma/6$ so the last inequality indeed holds by Lemma \ref{l:integeral}.
\end{proof}	

\noindent 
From the last two lemmas we immediately obtain:
\begin{corollary}\label{coro:success}
	With positive probability, all stages are successful. \qed
\end{corollary}

\section{$f(4)=4$}

Here we provide a computer-assisted proof that $f(4)=4$. Recall from the introduction that $f(4) \ge 4$
and recall from the first paragraph in Section 2 that to prove $f(4) \le 4$ it suffices to prove that every $[4,4,1]$-graph has a factor of independent transversals.

Consider the set of $[4,4,1]$-graphs with vertex parts
$V_1,V_2,V_3,V_4$ and $V_i=\{x_{i,1},x_{i,2},x_{i,3},x_{i,4}\}$.
As adding edges only makes the problem more difficult, we may assume that each of the $6$ pairs $(V_i,V_j)$
with $1 \le i < j \le 4$ induces a perfect matching $M_{i,j}$. By relabeling vertices, we may assume that
the perfect matching $(V_1,V_j)$ for $j=2,3,4$ is
$$
M_{1,j}=\{(x_{1,1},x_{j,1}),(x_{1,2},x_{j,2}),(x_{1,3},x_{j,3}),
(x_{1,4},x_{j,4})\}
$$ as the $12$ edges of these three matching are a spanning forest.
Thus, the set of graphs that we must construct and check correspond to constructing the remaining three perfect matchings
$M_{2,3}$, $M_{2,4}$, $M_{3,4}$. There are $4!=24$ choices for each of them, each corresponding to a permutation of one of the vertex classes involved in the matching. Overall, there are $24^3=13824$ graphs constructed this way (some may be isomorphic).

Checking each constructed graph for a factor of independent transversals
proceeds as follows. Take three permutations $\pi_2,\pi_3,\pi_4$ of $V_2,V_3,V_4$ respectively, and check if
the following factor $\{ \{x_{1,j},x_{2,\pi_2(j)},x_{3,\pi_3(j)},x_{4,\pi_4(j)} \} \,|\, j=1,2,3,4\}$
is a factor of independent transversals. Note that there are only $(4!)^3=13824$ choices for $\pi_2,\pi_3,\pi_4$, so there are only $13824$ checks to perform for each of the $13824$ constructed graphs.
It turns out that each of these graphs has a factor of independent transversals. A simple program implementing this search and verifying its conclusion appears in Appendix A.

\newpage

\appendix

\section{Source code determining $f(4)$}

\lstset{language=C++,tabsize=2}

\small{
\begin{lstlisting}

const int N = 4; const int F = 24;
int permutations[F][N] = {
	{0,1,2,3},{1,0,2,3},{2,0,1,3},{0,2,1,3},{1,2,0,3},{2,1,0,3},
	{2,1,3,0},{1,2,3,0},{3,2,1,0},{2,3,1,0},{1,3,2,0},{3,1,2,0},
	{3,0,2,1},{0,3,2,1},{2,3,0,1},{3,2,0,1},{0,2,3,1},{2,0,3,1},
	{1,0,3,2},{0,1,3,2},{3,1,0,2},{1,3,0,2},{0,3,1,2},{3,0,1,2} };

int g[N][N][N][N];

bool check()
{
	int p[N]; // p[i] determines the order of vertices of set i
	p[0] = 0;
	
	for (p[1] = 0; p[1] < F; p[1]++) 
		for (p[2] = 0; p[2] < F; p[2]++) 
			for (p[3] = 0; p[3] < F; p[3]++)
			{
				bool current = true;
				for (int i = 0; i < N; i++)
					for (int j = 0; j < N; j++)
						for (int k = j + 1; k < N; k++)
							if (g[j][permutations[p[j]][i]][k][permutations[p[k]][i]] == 1)
								current = false;
				if (current)
					return true;
			}
	return false;
}

int main()
{
	for (int i = 0; i < N; i++)
		for (int j = 0; j < N; j++)
			for (int k = 0; k < N; k++)
				for (int l = 0; l < N; l++)
					g[i][j][k][l] = 0;
	for (int i = 0; i < N; i++)
		for (int k = 1; k < N; k++)
			g[0][i][k][i] = 1;
	int p[7];
	for (p[1]=0; p[1] < F; p[1]++) // p1 determines matching between sets 1,2
		for (p[2] = 0; p[2] < F; p[2]++) // p2 determines matching between 1,3
			for (p[3] = 0; p[3] < F; p[3]++) // p3 determines matching between 2,3
			{
				/* reset the graph */
				for (int i = 1; i < N; i++)
					for (int j = 0; j < N; j++)
						for (int k = 1; k < N; k++)
							for (int l = 0; l < N; l++)
								g[i][j][k][l] = 0;
				g[1][0][2][permutations[p[1]][0]] = 1;  
				g[1][1][2][permutations[p[1]][1]] = 1;
				g[1][2][2][permutations[p[1]][2]] = 1;
				g[1][3][2][permutations[p[1]][3]] = 1; 
				g[1][0][3][permutations[p[2]][0]] = 1; 
				g[1][1][3][permutations[p[2]][1]] = 1; 
				g[1][2][3][permutations[p[2]][2]] = 1; 
				g[1][3][3][permutations[p[2]][3]] = 1; 
				g[2][0][3][permutations[p[3]][0]] = 1; 
				g[2][1][3][permutations[p[3]][1]] = 1; 
				g[2][2][3][permutations[p[3]][2]] = 1; 
				g[2][3][3][permutations[p[3]][3]] = 1; 
				if (!check())
					printf("failed\n,");
			}
	printf("Ended\n");
	return 0;
}
\end{lstlisting}
}
\end{document}